\documentclass[12pt,reqno]{amsart}

\usepackage{amsmath,amsthm,amscd,amsfonts,amssymb,mathtools}
\usepackage{mathrsfs}
\usepackage{graphicx,xcolor}
\usepackage{enumitem}
\usepackage[bookmarksnumbered,colorlinks,plainpages]{hyperref}

\newcommand{\Cp}{\mathbb C^+}
\newcommand{\R}{\mathbb R}
\newcommand{\C}{\mathbb C}

\renewcommand{\d}{\,\mathrm{d}}

\newcommand{\Pp}{\mathcal P}
\newtheorem{theorem}{Theorem}[section]
\newtheorem{lemma}[theorem]{Lemma}
\newtheorem{proposition}[theorem]{Proposition}
\newtheorem{corollary}[theorem]{Corollary}
\theoremstyle{definition}

\theoremstyle{remark}

\newcommand{\boxp}{\boxplus}
\newcommand{\Var}{\operatorname{Var}}
\newcommand{\Law}{\operatorname{Law}}
\newcommand{\supp}{\operatorname{supp}}
\newcommand{\dd}{\,\mathrm d}

\numberwithin{equation}{section}

\begin{document}
\setcounter{page}{1}


\title[Higher-Order Convex Orders under Free Additive Convolution]{Higher-Order Convex Orders under Free Additive Convolution}

\author[Cheng-Xuan Soong]{Cheng-Xuan Soong}

\address{
National Center for Theoretical Sciences, Mathematics Division,
National Taiwan University, Taipei 106, Taiwan
}

\email{cxsoong@ncts.ntu.edu.tw}
\date{\today}

\begin{abstract}
We prove that free additive convolution preserves the $k$-convex
order for every $k\ge2$, assuming finite $(k-1)$st absolute moments
for the compared measures and a finite $\max\{2,k-1\}$th absolute
moment for the convolving measure.
In particular, the case $k=4$ gives an affirmative answer to a
question of Hein\"avaara concerning fourth-order convolution
comparison.

\medskip

\end{abstract}

\maketitle

\section{Introduction}

Let $\Pp_c(\mathbb R)$ denote the set of compactly supported
Borel probability measures on $\mathbb R$, and let
$\Pp_k(\mathbb R)$ be the set of probability measures with finite
$k$th absolute moment, with the convention that
$\Pp_0(\mathbb R)$ is the set of all Borel probability measures
on $\mathbb R$. We write $\ast$ and $\boxp$ for the
classical and free additive convolutions, respectively. Free additive
convolution was introduced by Voiculescu \cite{Voi86}; see also
\cite{BV93,NS06}.

Higher-order convex orders form a natural hierarchy of comparison
relations between probability measures. We follow the formulation of
Denuit, Lef\`evre, and Shaked \cite{DLS98}; see
Subsection~\ref{subs:ho} for the precise definition. The cases $k=1$
and $k=2$ are the usual stochastic order and convex order,
respectively.

A basic question is whether such comparison relations are preserved
under convolution. In classical probability, the $k$-convex order is
preserved under classical convolution
\cite[Proposition~3.11]{DLS98}. The purpose of this paper is to
establish the corresponding result for free additive convolution under
natural moment assumptions.

Our main result is the following.

\begin{theorem}\label{thm:main1}
Let $k\ge2$. Let $\mu_1,\mu_2\in\Pp_{k-1}(\R)$ and let $\nu\in\Pp_{\max\{2,k-1\}}(\R)$. If $\mu_1\prec_{k\text{-cx}}\mu_2$, then
\[
    \mu_1\boxp\nu
    \prec_{k\text{-cx}}
    \mu_2\boxp\nu.
\]
\end{theorem}
The case $k=1$ was already established for free additive convolution in
\cite[Proposition~4.16]{BV93}.
The case $k=4$ is particularly relevant to recent work of Hein\"avaara
\cite{Hein26}. He introduced a fourth-order comparison relation, denoted
by $\prec_4$, which we recall in Subsection~2.1. This relation arises
naturally in the comparison between classical and free convolution:
the first three moments of $\mu\ast\nu$ and $\mu\boxplus\nu$ agree,
whereas their fourth moments need not. For compactly supported
probability measures, Hein\"avaara proved
$\mu\boxplus\nu\prec_4\mu\ast\nu$
\cite[Theorem~1.3]{Hein26}. He also showed that $\prec_4$ is preserved
under classical convolution \cite[Proposition~5.6]{Hein26}: if
$\mu_1,\mu_2,\nu\in\mathcal P_c(\mathbb R)$ and
$\mu_1\prec_4\mu_2$, then
$\mu_1\ast\nu\prec_4\mu_2\ast\nu$.
He then asked whether the analogous preservation property holds for
free additive convolution, namely whether
$\mu_1\boxplus\nu\prec_4\mu_2\boxplus\nu$ follows from
$\mu_1\prec_4\mu_2$ \cite[Question~5.8]{Hein26}.

We answer this question affirmatively, in fact under the weaker
assumption of finite third absolute moments. Under this assumption, we show that
the relation $\prec_4$ coincides with the $4$-convex order.
Consequently, Theorem~\ref{thm:main1} yields the following result.
\begin{corollary}\label{cor:main2}
Let $\mu_1,\mu_2,\nu\in\Pp_3(\R)$. If $\mu_1\prec_4\mu_2$, then
\[
    \mu_1\boxp\nu
    \prec_4
    \mu_2\boxp\nu.
\]
\end{corollary}

The main idea of the proof is to use the characterization of the
$k$-convex order by matching moments through order $k-1$ together
with a pointwise comparison of the potentials $U_k^\mu$.
The preservation of the moment conditions is standard, while the main
difficulty is to establish the corresponding comparison for the
potentials after free additive convolution. For this, we interpolate
between $\mu_1$ and $\mu_2$ and differentiate along the interpolation.
This produces a higher-order subordination kernel, whose key feature
is that it can be represented as the $(k+1)$st-order Cauchy transform
of a probability measure. A localization estimate and a uniqueness
theorem for higher-order Cauchy transforms then yield the desired
comparison of the convolved potentials.

The paper is organized as follows. Subsection~\ref{subs:ho} recalls the
$k$-convex order and Hein\"avaara's fourth-order relation, while
Subsection~\ref{sub:an} collects the analytic tools from free probability
needed below. Section~\ref{sec:ukmu} develops the higher-order potentials and
establishes their basic properties. Section~\ref{sec:subo} constructs the positive
subordination kernels and proves the required localization estimate.
Section~\ref{pf} completes the proof of Theorem~\ref{thm:main1}.

\section{Preliminaries} Throughout the remainder of the paper, $k\ge2$ is an integer. 
\subsection{Higher-order convex orders}\label{subs:ho}
Following \cite{DLS98}, we say that a function $f:\R\to\R$ is
$k$-convex if
\[
\begin{vmatrix}
1 & 1 & \cdots & 1\\
x_0 & x_1 & \cdots & x_k\\
\vdots & \vdots & \ddots & \vdots\\
x_0^{k-1} & x_1^{k-1} & \cdots & x_k^{k-1}\\
f(x_0) & f(x_1) & \cdots & f(x_k)
\end{vmatrix}
\ge0
\]
for every $x_0<\cdots<x_k$. For $f\in C^k(\R)$, $k$-convexity is equivalent to
$f^{(k)}\ge0$. For $\mu_1,\mu_2\in\Pp_0(\R)$, we write $\mu_1\prec_{k\text{-cx}}\mu_2$ if
\[
    \int_\R f(x)\dd\mu_1(x)
    \le
    \int_\R f(x)\dd\mu_2(x)
\]
for every $k$-convex function $f:\R\to\R$ for which both integrals
exist.  

We also recall the fourth-order relation considered by Hein\"avaara
in \cite{Hein26}. For $\mu_1,\mu_2\in\Pp_3(\R)$, we write $\mu_1\prec_4\mu_2$ if
\[
    \int_\R f(x)\dd\mu_1(x)
    \le
    \int_\R f(x)\dd\mu_2(x)
\]
for every $f\in C^4(\R)$ such that $f^{(4)}\ge0$ and $\supp f^{(4)}$ is compact.
\subsection{Cauchy transforms and subordination}
\label{sub:an}

Let $\mu\in\Pp_0(\R)$. We denote its Cauchy transform and reciprocal Cauchy transform by
\[
    G_\mu(z)
    :=
    \int_\R \frac{\dd\mu(x)}{z-x},
    \quad
    F_\mu(z)
    :=
    \frac{1}{G_\mu(z)},
    \quad z\in\Cp.
\]
Let $\nu\in\Pp_0(\R)$. Then
\cite[Theorem~4.1]{BB07} shows that there exist unique analytic
functions $\omega_\mu,\omega_\nu:\Cp\to\Cp$, called the
subordination functions, such that
\[F_{\mu\boxp\nu}(z)
    =
    F_\mu(\omega_\mu(z))
    =
    F_\nu(\omega_\nu(z)),
    \quad z\in\Cp,\]
\begin{equation}\label{eq:subordination-sum}
    \omega_\mu(z)+\omega_\nu(z)-z
    =
    F_{\mu\boxp\nu}(z),
\end{equation}
and
\begin{equation}\label{eq:subordination-normalization}
    \Im\omega_\mu(z)\ge\Im z,
    \quad
    \Im\omega_\nu(z)\ge\Im z,
    \quad z\in\Cp,
\end{equation}
together with
\[
    \frac{\omega_\mu(z)}{z}\longrightarrow1,
    \quad
    \frac{\omega_\nu(z)}{z}\longrightarrow1
    \quad\text{as }z\to\infty\text{ nontangentially}.
\]

We now specialize to measures with finite second moment.
By \cite[Proposition~2.2]{Maa92}, after translation, $\mu\in\Pp_2(\R)$ and has mean $m$ if and only if
\[F_\mu(z)
    =
    z-m-\int_\R\frac{\dd\tau_\mu(s)}{z-s},
    \quad z\in\Cp,\]
for some finite positive measure $\tau_\mu$, in which case $\tau_\mu(\R)=\Var(\mu)$. In particular,
\begin{equation}\label{eq:finite-variance-bound}
    |F_\mu(z)-z+m|
    \le
    \frac{\Var(\mu)}{\Im z},
    \quad z\in\Cp.
\end{equation}
Conversely, if for some $m\in\R$ and $C<\infty$,
\[
    |F_\mu(z)-z+m|
    \le
    \frac{C}{\Im z},
    \quad z\in\Cp,
\]
then $\mu\in\Pp_2(\R)$ has mean $m$ and $\Var(\mu)\le C$.


We conclude this subsection with the following uniqueness lemma.
\begin{lemma}\label{lem:uniqueness}
Let $M$ be a signed Radon measure on $\R$ such that
\[
    |M|([-R,R])=o(R)
    \quad\text{as } R\to\infty.
\]
Then, for every $z\in\Cp$,
\[
    \int_\R
    \frac{\dd |M|(s)}{|z-s|^{k+1}}
    <\infty.
\]
In particular, for every $z\in\Cp$, the function
$s\mapsto (z-s)^{-k-1}$ is integrable with respect to $M$.

Moreover, if
\[
    \int_\R
    \frac{\dd M(s)}{(z-s)^{k+1}}
    =0
    \quad\text{for all } z\in\Cp,
\]
then $M=0$.
\end{lemma}

\begin{proof}[\sc Proof.]
We first prove the asserted integrability. Fix $z\in\Cp$.
Since $|M|([-R,R])=o(R)$, we may choose $R_0>2|z|$ such that
\[
    |M|([-R,R])\le R
    \quad\text{for } R\ge R_0.
\]
On the one hand,
\[
    \int_{|s|\le R_0}
    \frac{\dd|M|(s)}{|z-s|^{k+1}}
    \le
    \frac{|M|([-R_0,R_0])}{(\Im z)^{k+1}}
    <\infty.
\]
On the other hand, for $n\ge0$, set
\[
    A_n:=\{s\in\R:2^nR_0<|s|\le2^{n+1}R_0\}.
\]
For $s\in A_n$, we have $|s|>2^nR_0>2|z|$, and hence
\[
    |z-s|
    \ge |s|-|z|
    \ge \frac{|s|}{2}
    >2^{n-1}R_0.
\]
Therefore,
\[\int_{A_n}\frac{\dd|M|(s)}{|z-s|^{k+1}}
    \le
    \frac{|M|([-2^{n+1}R_0,2^{n+1}R_0])}
         {(2^{n-1}R_0)^{k+1}} \le
    \frac{2^{n+1}R_0}
         {(2^{n-1}R_0)^{k+1}} =
    \frac{2^{k+2}}{R_0^k}2^{-nk}.\]
Since $k\ge2$, the series $\sum_{n\ge0}2^{-nk}$ converges, and the asserted integrability follows.

We now prove uniqueness. Let
\[
    P_y(x)
    :=
    \frac{1}{\pi}\frac{y}{x^2+y^2},
    \quad y>0,
\]
be the Poisson kernel.  We claim that, for every $\varphi\in C_c^\infty(\R)$,
\[
    \int_\R \varphi(x)(P_y*M)(x)\dd x
    \longrightarrow
    \int_\R \varphi(s)\dd M(s)
    \quad\text{as }y\downarrow0.
\]
Choose $A>0$ such that $\supp\varphi\subset[-A,A]$. For $|s|>2A$ and $0<y\le1$, we have
\[\int_\R |\varphi(x)|P_y(x-s)\dd x
    =
    \frac{1}{\pi}
    \int_{-A}^A
    \frac{y|\varphi(x)|}{(x-s)^2+y^2}\dd x \le
    \frac{4\|\varphi\|_{L^1}}{\pi|s|^2}.\]
Since $|M|$ is finite on compact sets and $|s|^{-2}$ is integrable with respect to $|M|$ outside a compact set by the dyadic-shell estimate, Tonelli's theorem gives
\[
    \int_{\R^2}
    |\varphi(x)|P_y(x-s)
    \dd(m\otimes|M|)(x,s)
    =
    \int_\R\int_\R
    |\varphi(x)|P_y(x-s)\dd x\dd|M|(s)<\infty.
\]
where $m$ denotes Lebesgue measure. Hence Fubini's theorem
applies, and
\[
    \int_\R \varphi(x)(P_y*M)(x)\dd x
    =
    \int_\R (P_y*\varphi)(s)\dd M(s).
\]
The Poisson kernels form an approximate identity, so $ P_y*\varphi\longrightarrow\varphi$. The estimate above provides an $|M|$-integrable bound outside a
compact set, uniformly for $0<y\le1$. Since $|M|$ is finite on compact sets, we may let $y\downarrow0$ to obtain
\[
    \int_\R (P_y*\varphi)(s)\dd M(s)
    \longrightarrow
    \int_\R \varphi(s)\dd M(s)
\]
as claimed. Using the assumption with $z=x+iy$, we obtain, for $x\in\R$ and $y>0$,
\[0
    =
    -\frac{1}{\pi}
    \Im\left(
        k!\int_\R
        \frac{\dd M(s)}{(x+iy-s)^{k+1}}
    \right) =
    (-1)^k
    \int_\R
    \partial_x^kP_y(x-s)\dd M(s)=
    (-1)^k
    \partial_x^k(P_y*M)(x).\]
    where we used
\[
    -\frac{1}{\pi}
    \Im\frac{k!}{(x+iy-s)^{k+1}}
    =
    (-1)^k\partial_x^kP_y(x-s).
\]
The differentiation under the integral sign is justified by the same
dyadic-shell estimate as above. Integration by parts gives
\[
    0
     =
    \int_\R
    \varphi(x)\partial_x^k(P_y*M)(x)\dd x =
    (-1)^k
    \int_\R
    \varphi^{(k)}(x)(P_y*M)(x)\dd x.
\]
Letting $y\downarrow0$ and applying the convergence above with
$\varphi^{(k)}$ in place of $\varphi$, we obtain
\[
    \int_\R \varphi^{(k)}(s)\dd M(s)=0.
\]

By repeated application of \cite[Theorem~1.47]{Legoll22},
there exists a polynomial $p$ of degree at most $k-1$ such that
\[
    \dd M(s)=p(s)\dd s,\quad |M|([-R,R])
    =
    \int_{-R}^R |p(s)|\dd s.
\] 
If $p\not\equiv0$, the right-hand side is not $o(R)$ as
$R\to\infty$, contradicting the assumption on $M$. Hence
$p=0$, and therefore $M=0$.
\end{proof}
\section{Higher-order potentials}\label{sec:ukmu}
For $\mu\in\Pp_{k-1}(\R)$, define
\[
    U_k^\mu(t)
    :=
    \frac{1}{(k-1)!}
    \int_\R (x-t)_+^{k-1}\dd\mu(x)
    =
    \frac{1}{(k-1)!}
    \int_{(t,\infty)}
    (x-t)^{k-1}\dd\mu(x),
    \quad t\in\R,
\]
where $x_+:=\max\{x,0\}.$
The finite $(k-1)$st absolute moment assumption implies that
$U_k^\mu(t)<\infty$ for every $t\in\R$, and dominated convergence shows that $U_k^\mu$ is continuous. 

The $k$-convex order admits the following characterization in terms
of these functions.

\begin{proposition}[{\cite[Theorem~3.2]{DLS98}}]
\label{prop:char}
Let $\mu_1,\mu_2\in\Pp_{k-1}(\R)$. Then $\mu_1\prec_{k\text{-cx}}\mu_2$
if and only if
\[
    \int_\R x^j\d\mu_1(x) = \int_\R x^j\d\mu_2(x),
    \quad j=1,\ldots,k-1,
\]
and
\[
    U_k^{\mu_1}(t)
    \le
    U_k^{\mu_2}(t)
    \quad \text{for all }t\in\R.
\]
\end{proposition}
The preceding characterization also identifies the fourth-order
relation of Hein\"avaara with the $4$-convex order.
\begin{proposition}\label{prop:Heinavaara-order}
Let $\mu,\nu\in\Pp_3(\R)$. Then $\mu\prec_4\nu$ if and only if $\mu\prec_{4\text{-cx}}\nu$.
\end{proposition}

\begin{proof}[\sc Proof.] The ``if'' part is immediate, since every function admissible in the
definition of $\prec_4$ satisfies $f^{(4)}\ge0$ and is therefore
$4$-convex.

Conversely, suppose that $\mu\prec_4\nu$. Every polynomial $p$ of
degree at most three, as well as $-p$, is admissible in the definition of $\prec_4$, since $p^{(4)}=0$. Hence
\[
    \int_\R x^j\dd\mu(x)
    =
    \int_\R x^j\dd\nu(x),
    \quad j=1,2,3.
\]

Let $\varphi\in C_c^\infty(\R)$ be nonnegative and define
\[
    f_\varphi(x)
    :=
    \frac{1}{3!}
    \int_{-\infty}^x (x-t)^3\varphi(t)\dd t.
\]
Then $f_\varphi\in C^4(\R)$ and $f_\varphi^{(4)}=\varphi\ge0$, so $f_\varphi$ is admissible.
By Fubini's theorem,
\[0\le\int_\R f_\varphi(x)\dd(\nu-\mu)(x)=\int_\R
    \varphi(t)
    (
        U_4^\nu(t)-U_4^\mu(t)
    )
    \dd t .\]
Since this holds for every nonnegative
$\varphi\in C_c^\infty(\R)$ and
$U_4^\nu-U_4^\mu$ is continuous, we obtain
\[
    U_4^\mu(t)\le U_4^\nu(t)
    \quad \text{for all }t\in\R.
\]
Proposition~\ref{prop:char} therefore yields $\mu\prec_{4\text{-cx}}\nu$.
\end{proof}

We next record two basic properties of the potentials $U_k^\mu$.
\begin{lemma}\label{lem:C0}
Let $\mu_1,\mu_2\in\Pp_{k-1}(\R)$ have matching moments through
order $k-1$. Then
\[
    U_k^{\mu_2}-U_k^{\mu_1}\in C_0(\R):=\{f:\R\to\R\text{ is continuous}:f(t)\longrightarrow 0\text{ as $|t|\to\infty$}\}.
\]
\end{lemma}

\begin{proof}[\sc Proof.]
Since both $U_k^{\mu_1}$ and $U_k^{\mu_2}$ are continuous, it remains only to verify their decay at infinity. As $t\to+\infty$, the integrability of $|x|^{k-1}$
with respect to $\mu_1+\mu_2$ yields
\begin{align*}
    |U_k^{\mu_2}(t)-U_k^{\mu_1}(t)|
    &\le
    \frac{1}{(k-1)!}
    \int_{(t,\infty)}
    |x-t|^{k-1}\dd(\mu_1+\mu_2)(x) \\
    &\le
    \frac{1}{(k-1)!}
    \int_{(t,\infty)}
    |x|^{k-1}\dd(\mu_1+\mu_2)(x)
    \longrightarrow0,
\end{align*}
where we used $|x-t|\le |x|$ for $x>t>0$. Then, as $t\to-\infty$,
\begin{align*}
    |U_k^{\mu_2}(t)-U_k^{\mu_1}(t)|
    &=
    \frac{1}{(k-1)!}
    \left|
        \int_{(t,\infty)}
        (x-t)^{k-1}\dd(\mu_2-\mu_1)(x)
    \right| \\
    &=
    \frac{1}{(k-1)!}
    \left|
        \int_{(-\infty,t]}
        (x-t)^{k-1}\dd(\mu_2-\mu_1)(x)
    \right| \\
    &\le
    \frac{1}{(k-1)!}
    \int_{(-\infty,t]}
    |x-t|^{k-1}\dd(\mu_1+\mu_2)(x) \\
    &\le
    \frac{1}{(k-1)!}
    \int_{(-\infty,t]}
    |x|^{k-1}\dd(\mu_1+\mu_2)(x)
    \longrightarrow0,
\end{align*}
where the second equality follows from the matching moments, and
we used $|x-t|\le |x|$ for $x\le t<0$. The final convergence again
follows from the integrability of $|x|^{k-1}$.
\end{proof}

\begin{lemma}\label{lem:UC}
Let $\mu\in\Pp_{k-1}(\R)$. Then
\[
    G_\mu(z)
    =
    k!\int_\R
    \frac{U_k^\mu(t)}{(z-t)^{k+1}}
    \dd t,\quad z\in\C^+.
\]
\end{lemma}

\begin{proof}[\sc Proof.]
Fix $z\in\Cp$. Since
$0< x-t\le |x|+|t|$ on $\{(x,t)\in\R^2:t<x\}$, Tonelli's theorem gives
\[
\begin{aligned}
&\int_{\{(x,t)\in\R^2:t<x\}}
\frac{(x-t)^{k-1}}{|z-t|^{k+1}}
\dd(\mu\otimes m)(x,t)\\
&\quad=
\int_\R\int_{-\infty}^x
\frac{(x-t)^{k-1}}{|z-t|^{k+1}}
\dd t\dd\mu(x) \\
&\quad\le
2^{k-2}
\int_\R\int_{-\infty}^x
\frac{|x|^{k-1}+|t|^{k-1}}
{|z-t|^{k+1}}
\dd t\dd\mu(x) \\
&\quad\le
2^{k-2}
\left[
\left(\int_\R |x|^{k-1}\dd\mu(x)\right)
\left(\int_\R\frac{\dd t}{|z-t|^{k+1}}\right)
+
\int_\R
\frac{|t|^{k-1}}{|z-t|^{k+1}}\dd t
\right]
<\infty.
\end{aligned}
\]
where $m$ denotes Lebesgue measure on $\R$. Hence Fubini's theorem applies, and 
\[
\begin{aligned}
    k!\int_\R
    \frac{U_k^\mu(t)}{(z-t)^{k+1}}
    \dd t
    &=
    \int_\R\int_t^\infty
    \frac{k(x-t)^{k-1}}{(z-t)^{k+1}}
    \dd\mu(x)\dd t \\
    &=
    \int_\R
    \left(
        k\int_{-\infty}^x
        \frac{(x-t)^{k-1}}{(z-t)^{k+1}}
        \dd t
    \right)\dd\mu(x) \\
    &=
    \int_\R
    \left(
        k\int_0^\infty
        \frac{u^{k-1}}{(z-x+u)^{k+1}}
        \dd u
    \right)\dd\mu(x) \quad (\text{change of variables
$u=x-t$}) \\
    &=
    \int_\R
    \frac{1}{z-x}
    \left[
        \left(
            \frac{u}{z-x+u}
        \right)^k
    \right]_{u=0}^{u=\infty}
    \dd\mu(x) \\
    &=
    \int_\R\frac{\dd\mu(x)}{z-x} =
    G_\mu(z),
\end{aligned}
\]
as claimed.
\end{proof}
\section{Positive subordination kernels}
\label{sec:subo}
In this section, we use analytic subordination to construct a family
of probability measures that will serve as positive kernels in the
proof of the main theorem.

Let
\[
    \lambda_r=(1-r)\mu_1+r\mu_2,
    \quad 0\le r\le1,
\]
and let $\omega_r:\Cp\to\Cp$ denote the subordination function such that
\[
    F_{\lambda_r\boxplus\nu}(z)
    =
    F_{\lambda_r}(\omega_r(z)),
    \quad z\in\Cp.
\]
The paths $r\mapsto\lambda_r$ and $r\mapsto\nu$ are differentiable,
with constant derivatives $\mu_2-\mu_1$ and $0$, respectively.
Hence \cite[Theorem~26]{BS12} gives
\[\partial_r G_{\lambda_r\boxplus\nu}(z)
    =
    \omega_r'(z)
    (
        G_{\mu_2}(\omega_r(z))
        -
        G_{\mu_1}(\omega_r(z))
    ).\]
Since $\lambda_0=\mu_1$ and $\lambda_1=\mu_2$, integrating over
$r\in[0,1]$ gives
\begin{equation}\label{eq:integrated-subordination}
    G_{\mu_2\boxplus\nu}(z)
    -
    G_{\mu_1\boxplus\nu}(z)
    =
    \int_0^1
    \omega_r'(z)
    (
        G_{\mu_2}(\omega_r(z))
        -
        G_{\mu_1}(\omega_r(z))
    )
    \dd r .
\end{equation}
Using Lemma~\ref{lem:UC} in
\eqref{eq:integrated-subordination}, we obtain
\[G_{\mu_2\boxplus\nu}(z)
    -
    G_{\mu_1\boxplus\nu}(z)
    =
    k!\int_0^1\int_\R
    (
        U_k^{\mu_2}(t)-U_k^{\mu_1}(t)
    )
    \frac{\omega_r'(z)}
         {(\omega_r(z)-t)^{k+1}}
    \dd t\dd r .\]

We then represent the kernel
\[
    \frac{\omega_r'(z)}
         {(\omega_r(z)-t)^{k+1}}
\]
as the $(k+1)$st-order Cauchy transform of a probability measure. For each fixed $r\in[0,1]$,
\cite[Theorem~3.1]{Biane98} gives a family of probability measures
$\{\rho_{r,t}:t\in\R\}$ satisfying
\begin{equation}\label{eq:subordination-measure}
    G_{\rho_{r,t}}(z)
    =
    \frac{1}{\omega_r(z)-t},
    \quad z\in\Cp.
\end{equation}
Differentiating \eqref{eq:subordination-measure} yields
\begin{equation}\label{eq:kernel-target-biane}
    \frac{\omega_r'(z)}
         {(\omega_r(z)-t)^{k+1}}
    =
    -G_{\rho_{r,t}}(z)^{k-1}
     G_{\rho_{r,t}}'(z).
\end{equation}

Now let $X_1,\ldots,X_k$ be independent random variables with common law $\rho_{r,t}$, and let
\[
    (\Theta_1,\ldots,\Theta_k)
    \sim \operatorname{Dirichlet}(1,\ldots,1)
\]
be independent of $X_1,\ldots,X_k$. Define
\[
    K_{r,t}^{(k)}
    :=
    \Law\left(
        \sum_{j=1}^k\Theta_jX_j
    \right).
\]

For fixed $x_1,\ldots,x_k\in\R$, the Hermite--Genocchi formula,
applied to $f_z(x)=(z-x)^{-1}$, gives
\[ \prod_{j=1}^k\frac{1}{z-x_j}
    =
    \mathbb E_\Theta
    \left[
        \frac{1}
        {\left(
            z-\sum_{j=1}^k\Theta_jx_j
        \right)^k}
    \right];\]
see \cite[Appendix~A, Equations.~(A.2)--(A.3)]{Lesch17}.
Taking expectation with respect to $X_1,\ldots,X_k$ gives
\[
    G_{\rho_{r,t}}(z)^k
    =
    \int_\R
    \frac{\dd K_{r,t}^{(k)}(s)}
         {(z-s)^k}.
\]
Differentiating both sides with respect to $z$ and cancelling the
common factor $-k$, we obtain
\begin{equation}\label{eq:UG-Cauchy-derivative}
    -G_{\rho_{r,t}}(z)^{k-1}
     G_{\rho_{r,t}}'(z)
    =
    \int_\R
    \frac{\dd K_{r,t}^{(k)}(s)}
         {(z-s)^{k+1}}.
\end{equation}
Combining \eqref{eq:kernel-target-biane} and
\eqref{eq:UG-Cauchy-derivative} yields
\begin{equation}\label{eq:higher-order-kernel-transform}
    \frac{\omega_r'(z)}
         {(\omega_r(z)-t)^{k+1}}
    =
    \int_\R
    \frac{\dd K_{r,t}^{(k)}(s)}
         {(z-s)^{k+1}},
    \quad z\in\Cp.
\end{equation}
We next verify that
$\{K_{r,t}^{(k)}\}_{(r,t)\in[0,1]\times\R}$
forms a Borel probability kernel.

\begin{lemma}\label{lem:mea}
For every bounded continuous function $f:\R\to\C$, the map
\[
    (r,t)
    \longmapsto
    \int_\R f(s)\dd K_{r,t}^{(k)}(s)
\]
is continuous on $[0,1]\times\R$. In particular,
$(r,t)\mapsto K_{r,t}^{(k)}(A)$ is Borel measurable for every
Borel set $A\subset\R$.
\end{lemma}

\begin{proof}[\sc Proof.]
Let $(r_n,t_n)\to(r,t)$. By the fixed-point characterization of the subordination functions
\cite[Theorem~3.2]{BB07} and the continuity theorem for
Denjoy--Wolff points \cite[Theorem~1.1]{BBH22}, $\omega_{r_n}(z)\longrightarrow\omega_r(z)$ for all $z\in\Cp$.
Hence
\[
    G_{\rho_{r_n,t_n}}(z)
    =
    \frac{1}{\omega_{r_n}(z)-t_n}
    \longrightarrow
    \frac{1}{\omega_r(z)-t}
    =
    G_{\rho_{r,t}}(z),
    \quad z\in\Cp.
\]
Therefore $\rho_{r_n,t_n}\Rightarrow\rho_{r,t}$ by \cite[Theorem~2.5]{Maa92}.

For $f\in C_b(\R)$, define
\[
    T_kf(x_1,\ldots,x_k)
    :=
    \mathbb E_\Theta
    f\left(\sum_{j=1}^k\Theta_jx_j\right).
\]
Then $T_kf\in C_b(\R^k)$, and by the definition of
$K_{r,t}^{(k)}$,
\[
    \int_\R f(s)\dd K_{r,t}^{(k)}(s)
    =
    \int_{\R^k}
    T_kf(x_1,\ldots,x_k)\,
    \dd\rho_{r,t}^{\otimes k}(x_1,\ldots,x_k).
\]
Since weak convergence is preserved under finite products, $\rho_{r_n,t_n}^{\otimes k}
    \Rightarrow
    \rho_{r,t}^{\otimes k}$, and therefore
\[
    \int_\R f(s)\dd K_{r_n,t_n}^{(k)}(s)
    \longrightarrow
    \int_\R f(s)\dd K_{r,t}^{(k)}(s).
\]
This proves the first assertion. The second follows because, for every
Borel set $A\subset\R$, the map $\mu\longmapsto\mu(A)$ is Borel measurable on $\Pp_0(\R)$ equipped with the weak topology.
\end{proof}

\begin{lemma}\label{lem:higher-order-kernel-localization} Let $\nu\in\Pp_{2}(\R)$. Set $v_k:=2\Var(\nu)/(k+1).$ For every $R>0$,
\[
    \sup_{r\in[0,1]}
    \int_\R K_{r,t}^{(k)}([-R,R])\dd t
    \le
    2R+4\sqrt{v_k}.
\]
\end{lemma}
\begin{proof}[\sc Proof.] We first determine the mean and bound the variance of
$K_{r,t}^{(k)}$.

Let $\widetilde\omega_r$ denote the subordination function such that $F_{\lambda_r\boxplus\nu}(z)
    =
    F_\nu(\widetilde\omega_r(z))$ for all $z\in\Cp$.
Together with \eqref{eq:subordination-sum}, this gives $\omega_r(z)-z
    =
    F_\nu(\widetilde\omega_r(z))
    -
    \widetilde\omega_r(z).$
Hence, by \eqref{eq:subordination-normalization} and \eqref{eq:finite-variance-bound},
\[
    \left|
        \omega_r(z)-z+\int_\R x\dd\nu(x)
    \right|
    \le
    \frac{\Var(\nu)}{\Im z},
    \quad z\in\C^+.
\]
By \eqref{eq:subordination-measure}, $F_{\rho_{r,t}}(z)=\omega_r(z)-t$ and therefore
\[
    \left|
        F_{\rho_{r,t}}(z)-z
        +t+\int_\R x\dd\nu(x)
    \right|
    \le
    \frac{\Var(\nu)}{\Im z}.
\]
By the converse part of \eqref{eq:finite-variance-bound},
\begin{equation}\label{eq:mv}
    \int_\R x\dd\rho_{r,t}(x)
    =
    t+\int_\R x\dd\nu(x)
    \quad\text{and}\quad
    \Var(\rho_{r,t})
    \le
    \Var(\nu).
\end{equation}
Let
\[
    S=\sum_{j=1}^k\Theta_jX_j,
\]
so that $K_{r,t}^{(k)}=\Law(S)$. By \eqref{eq:mv}, conditionally on
$\Theta=(\Theta_1,\ldots,\Theta_k)$,
\[
     \mathbb E(S\mid\Theta)
    =
    t+\int_\R x\dd\nu(x)\quad\text{and}\quad \Var(S\mid\Theta)
    =
    \Var(\rho_{r,t})
    \sum_{j=1}^k\Theta_j^2.
\]
Using $K_{r,t}^{(k)}=\Law(S)$ and the conditional formulas above, we obtain
\[
\begin{aligned}
    \int_\R s\dd K_{r,t}^{(k)}(s)
    &=
    \mathbb E S
    =
    \mathbb E[\mathbb E(S\mid\Theta)]
    =
    t+\int_\R x\dd\nu(x), \\
    \Var(K_{r,t}^{(k)})
    &=
    \Var(S)
    =
    \mathbb E[\Var(S\mid\Theta)] =
    \Var(\rho_{r,t})
    \mathbb E\left[\sum_{j=1}^k\Theta_j^2\right]
    =
    \frac{2}{k+1}\Var(\rho_{r,t})
    \le v_k,
\end{aligned}
\]
where we used that $\mathbb E(S\mid\Theta)$ is constant and
$\mathbb E\Theta_j^2=2/(k(k+1))$.

Now set $a:=t+\int_\R x\dd\nu(x)$. If $|a|>R$, then $[-R,R]
    \subset
    \{s\in\R:|s-a|\ge |a|-R\}$. Hence Chebyshev's inequality yields 
\[
    K_{r,t}^{(k)}([-R,R])
    \le
    \frac{v_k}{(|a|-R)^2}.
\]
Together with the trivial bound
$K_{r,t}^{(k)}([-R,R])\le1$, we have
\[
    K_{r,t}^{(k)}([-R,R])
    \le
    \min\left\{
        1,
        \frac{v_k}{(|a|-R)^2}
    \right\}.
\]

If $v_k=0$, then $ K_{r,t}^{(k)}
    =
    \delta_{\,t+\int_\R x\dd\nu(x)}$,
and hence
\[
    \int_\R K_{r,t}^{(k)}([-R,R])\dd t
    =
    2R.
\]

Suppose now that $v_k>0$. Splitting according to
\[
    \left|
        t+\int_\R x\dd\nu(x)
    \right|
    \le
    R+\sqrt{v_k},
\]
and using the preceding bound on the complement, we obtain
\[
\begin{aligned}
    \int_\R K_{r,t}^{(k)}([-R,R])\dd t
    \le
    2(R+\sqrt{v_k})
    +
    2\int_{R+\sqrt{v_k}}^\infty
    \frac{v_k}{(s-R)^2}\dd s
    =
    2R+4\sqrt{v_k}
\end{aligned}
\]
as needed.
\end{proof}
\section{Proof of Theorem~\ref{thm:main1}}\label{pf}
We now prove Theorem~\ref{thm:main1}. By Proposition~\ref{prop:char}, $\mu_1$ and $\mu_2$ have matching moments through order $k-1$, and  $U_k^{\mu_1}(t)\le U_k^{\mu_2}(t)$ for all $t\in\R$.

\medskip
\noindent\textbf{Claim} $(\star)$. $U_k^{\mu_1\boxp\nu}(t)
    \le
    U_k^{\mu_2\boxp\nu}(t)$ for all $t\in\R$.
\begin{proof}[\sc Proof of Claim~$(\star)$]
For every Borel set $A\subset\R$, define
\begin{equation}\label{eq:Mdef}
    M_k(A)
    :=
    \int_0^1\int_\R
    (
        U_k^{\mu_2}(t)-U_k^{\mu_1}(t)
    )
    K_{r,t}^{(k)}(A)\dd t\dd r
\end{equation}
By Lemma~\ref{lem:mea} and the nonnegativity of
$U_k^{\mu_2}-U_k^{\mu_1}$, \eqref{eq:Mdef} defines a positive Borel measure on $\R$. By Lemma~\ref{lem:C0}, $U_k^{\mu_2}-U_k^{\mu_1}\in C_0(\R)$. Given $\varepsilon>0$, choose $T>0$ such that
\[
    |
        U_k^{\mu_2}(t)-U_k^{\mu_1}(t)
    |
    <\varepsilon
    \quad \text{whenever }|t|>T.
\]
Then, for every $R>0$, by $K_{r,t}^{(k)}([-R,R])\le1$ on $[-T,T]$
and Lemma~\ref{lem:higher-order-kernel-localization}, we have
\[
\begin{aligned}
    M_k([-R,R])
    &=
    \int_0^1\int_{|t|\le T}
    (
        U_k^{\mu_2}(t)-U_k^{\mu_1}(t)
    )
    K_{r,t}^{(k)}([-R,R])
    \dd t\dd r \\
    &\quad+
    \int_0^1\int_{|t|>T}
    (
        U_k^{\mu_2}(t)-U_k^{\mu_1}(t)
    )
    K_{r,t}^{(k)}([-R,R])
    \dd t\dd r \\
    &\le
    2T
    \|
        U_k^{\mu_2}-U_k^{\mu_1}
    \|_\infty
    +
    \varepsilon
    \int_0^1\int_{|t|>T}
    K_{r,t}^{(k)}([-R,R])
    \dd t\dd r \\
    &\le
    2T
    \|
        U_k^{\mu_2}-U_k^{\mu_1}
    \|_\infty
    +
    \varepsilon
    \left(
        2R+4\sqrt{v_k}
    \right).
\end{aligned}
\]
In particular, $M_k$ is finite on compact sets and hence is a Radon measure. Moreover,
\[
    \limsup_{R\to\infty}
    \frac{M_k([-R,R])}{R}
    \le 2\varepsilon.
\]
Since $\varepsilon>0$ is arbitrary, we obtain $M_k([-R,R])=o(R)$ as $R\to\infty$. By Lemma~\ref{lem:uniqueness},
$(z-s)^{-k-1}$ is integrable with respect to $M_k$. Hence
\begin{equation}\label{eq:Mtransform}
\begin{aligned}
    k!\int_\R
    \frac{\dd M_k(s)}{(z-s)^{k+1}}
    &=
    k!\int_0^1\int_\R
    (
        U_k^{\mu_2}(t)-U_k^{\mu_1}(t)
    )
    \left(
        \int_\R
        \frac{\dd K_{r,t}^{(k)}(s)}
             {(z-s)^{k+1}}
    \right)
    \dd t\dd r
    \\
    &=
    k!\int_0^1\int_\R
    (
        U_k^{\mu_2}(t)-U_k^{\mu_1}(t)
    )
    \frac{\omega_r'(z)}
         {(\omega_r(z)-t)^{k+1}}
    \dd t\dd r
    \\
    &=
    \int_0^1
    \omega_r'(z)
    (
        G_{\mu_2}(\omega_r(z))
        -
        G_{\mu_1}(\omega_r(z))
    )
    \dd r
    \\
    &=
    G_{\mu_2\boxplus\nu}(z)
    -
    G_{\mu_1\boxplus\nu}(z),
    \quad z\in\Cp.
\end{aligned}
\end{equation}
Here the first equality follows from \eqref{eq:Mdef} and
Fubini's theorem, the second
from \eqref{eq:higher-order-kernel-transform}, the third from
Lemma~\ref{lem:UC} applied at
$w=\omega_r(z)$, and the last from
\eqref{eq:integrated-subordination}.

On the other hand, by the Minkowski inequality for noncommutative
$L^p$-spaces \cite[Proposition~4.40]{Hiai21}, the assumptions
$\mu_1,\mu_2\in\Pp_{k-1}(\R)$ and
$\nu\in\Pp_{\max\{2,k-1\}}(\R)$ imply that $\mu_1\boxp\nu,\mu_2\boxp\nu\in\Pp_{k-1}(\R)$. Lemma~\ref{lem:UC} then yields
\begin{equation}\label{eq:outputpotentialtransform}
    G_{\mu_2\boxp\nu}(z)-G_{\mu_1\boxp\nu}(z)
    =
    k!\int_\R
    \frac{
        U_k^{\mu_2\boxp\nu}(t)-U_k^{\mu_1\boxp\nu}(t)
    }{(z-t)^{k+1}}
    \dd t,
    \quad z\in\C^+.
\end{equation}
Moreover, since $\mu_1$ and $\mu_2$ have matching moments through
order $k-1$, the moment--cumulant formula
\cite[Proposition~11.4]{NS06} and the additivity of free cumulants
\cite[Proposition~12.3]{NS06} imply that
$\mu_1\boxp\nu$ and $\mu_2\boxp\nu$ also have matching moments through order $k-1$. Thus Lemma~\ref{lem:C0} gives $U_k^{\mu_2\boxp\nu}-U_k^{\mu_1\boxp\nu}\in C_0(\R)$ and therefore
\[
    \int_{-R}^R
    |
        U_k^{\mu_2\boxp\nu}(t)-U_k^{\mu_1\boxp\nu}(t)
    |
    \dd t
    =
    o(R)\quad\text{as $R\to\infty$}.
\]
Comparing \eqref{eq:Mtransform} and \eqref{eq:outputpotentialtransform},
and applying Lemma~\ref{lem:uniqueness} to the signed Radon measure
\[
    \dd\widetilde M(t)
    :=
    \dd M_k(t)
    -
    (
        U_k^{\mu_2\boxp\nu}(t)-U_k^{\mu_1\boxp\nu}(t)
    )\dd t,
\]
which satisfies $|\widetilde M|([-R,R])=o(R)$, we obtain $\widetilde M=0$. Hence $\dd M_k(t)
    =
    (
        U_k^{\mu_2\boxp\nu}(t)-U_k^{\mu_1\boxp\nu}(t)
    )\dd t.$
Therefore, the positivity of $M_k$ implies that
\[
    U_k^{\mu_2\boxp\nu}(t)-U_k^{\mu_1\boxp\nu}(t)\ge0\quad\text{for almost every $t\in\R$.}
\]
By continuity, the inequality holds
for every $t\in\R$, as claimed.
\end{proof}
By Claim~$(\star)$, we have $U_k^{\mu_1\boxplus\nu}(t)
\le
U_k^{\mu_2\boxplus\nu}(t)$ for all $t\in\R$.
Moreover, as shown above, $\mu_1\boxplus\nu$ and $\mu_2\boxplus\nu$ have matching moments up to order $k-1$. Therefore, Proposition~\ref{prop:char} yields $\mu_1\boxplus\nu\prec_{k\text{-cx}}\mu_2\boxplus\nu,$
as desired.\qed

\medskip
In particular, taking $k=4$ and applying
Proposition~\ref{prop:Heinavaara-order} yields
Corollary~\ref{cor:main2}.


\begin{thebibliography}{99}


\bibitem[BB07]{BB07}
S.~T. Belinschi and H.~Bercovici,
\newblock A new approach to subordination results in free probability,
\newblock \emph{J. Anal. Math.} \textbf{101} (2007), 357--365.

\bibitem[BBH22]{BBH22}
S.~T. Belinschi, H.~Bercovici, and C.-W.~Ho,
\newblock On the convergence of Denjoy--Wolff points,
\newblock arXiv:2203.16728, 2022.

\bibitem[Bia98]{Biane98}
P.~Biane,
\newblock Processes with free increments,
\newblock \emph{Math. Z.} \textbf{227} (1998), no.~1, 143--174.

\bibitem[BS12]{BS12}
S.~T. Belinschi and D.~Shlyakhtenko,
\newblock Free probability of type B: analytic interpretation and applications,
\newblock \emph{Amer. J. Math.} \textbf{134} (2012), no.~1, 193--234.

\bibitem[BV93]{BV93}
H.~Bercovici and D.~Voiculescu,
\newblock Free convolution of measures with unbounded support,
\newblock \emph{Indiana Univ. Math. J.} \textbf{42} (1993), no.~3, 733--773.

\bibitem[DLS98]{DLS98}
M.~Denuit, C.~Lef\`evre, and M.~Shaked,
\newblock The $s$-convex orders among real random variables, with applications,
\newblock \emph{Math. Inequal. Appl.} \textbf{1} (1998), no.~4, 585--613.


\bibitem[Hei26]{Hein26}
O.~Hein\"avaara,
\newblock Convolution comparison measures,
\newblock arXiv:2602.10373v1, 2026.

\bibitem[Hia21]{Hiai21}
F.~Hiai,
\textit{Lectures on Selected Topics in von Neumann Algebras},
EMS Series of Lectures in Mathematics,
EMS Press, Berlin, 2021.

\bibitem[Leg22]{Legoll22}
F.~Legoll,
\newblock \emph{Partial Differential Equations: Variational Approaches},
\newblock lecture notes, 2022.

\bibitem[Les17]{Lesch17}
M.~Lesch,
\newblock Divided differences in noncommutative geometry:
Rearrangement Lemma, functional calculus and expansional formula,
\newblock \emph{J. Noncommut. Geom.} \textbf{11} (2017), no.~1, 193--223.

\bibitem[Maa92]{Maa92}
H.~Maassen,
\newblock Addition of freely independent random variables,
\newblock \emph{J. Funct. Anal.} \textbf{106} (1992), no.~2, 409--438.

\bibitem[NS06]{NS06}
A.~Nica and R.~Speicher,
\newblock \emph{Lectures on the Combinatorics of Free Probability},
\newblock London Mathematical Society Lecture Note Series, vol.~335,
Cambridge University Press, Cambridge, 2006.


\bibitem[Voi86]{Voi86}
D.~Voiculescu,
\newblock Addition of certain non-commuting random variables,
\newblock \emph{J. Funct. Anal.} \textbf{66} (1986), no.~3, 323--346.

\end{thebibliography}
\end{document}